\documentclass[12pt,reqno]{amsart}
\usepackage{a4wide}
\usepackage{amsthm, amssymb, amstext, amscd, amsfonts, amsxtra, latexsym, amsmath, comment, tikz,
	mathrsfs, mathtools, esint, stmaryrd, cancel,scalerel}
\usepackage{color}
\usepackage[all]{xy}
\usepackage{fancybox}
\usepackage{fullpage}
\usepackage[english]{babel}
\usepackage[latin1]{inputenc}
\usepackage{enumitem}
\usepackage{mdframed}
\usepackage{stackengine}
\usepackage{scalerel}
\usepackage{scalerel,stackengine}
\usepackage{graphicx} 
\stackMath
\newcommand\reallywidehat[1]{%
	\savestack{\tmpbox}{\stretchto{%
			\scaleto{%
				 \scalerel*[\widthof{\ensuremath{#1}}]{\kern-.6pt\bigwedge\kern-.6pt}%
				{\rule[-\textheight/2]{1ex}{\textheight}}
			}{\textheight}%
		}{0.5ex}}%
	\stackon[1pt]{#1}{\tmpbox}%
}
\numberwithin{equation}{section}
\newtheorem{theorem}{Theorem}
\numberwithin{theorem}{section}
\newtheorem{lemma}[theorem]{Lemma}
\newtheorem{proposition}[theorem]{Proposition}
\newtheorem{corollary}[theorem]{Corollary}

\theoremstyle{remark}
\newtheorem{remark}{Remark}

\theoremstyle{definition}
\newtheorem*{definition}{Definition}

\newcommand{\Z}{\mathbb{Z}}
\newcommand{\N}{\mathbb{N}}
\newcommand{\Q}{\mathbb{Q}}

\newcommand{\C}{\mathbb{C}}

\newcommand{\im}{\operatorname{Im}}
\renewcommand{\Im}{\operatorname{Im}}
\newcommand{\sgn}{\operatorname{sgn}}

\renewcommand{\H}{\mathbb{H}}
\newcommand{\SL}{\text{\rm SL}}

\makeatletter
\newcommand{\vast}{\bBigg@{4}}
\newcommand{\Vast}{\bBigg@{5}}
\makeatother

\renewcommand{\b}[1]{\boldsymbol{#1}}

\renewcommand{\pmod}[1]{\  \,  \left( \rm{mod} \,  #1 \right)}

\newlength{\parenheight}
\newlength{\parendepth}
\newlength{\parendrop}

\newcommand{\paren}[4]{%
	\settoheight{\parenheight}{\(#4 #2\)}%
	\settodepth{\parendepth}{\(#4 #2\)}
	\addtolength{\parendepth}{.5ex}
	\addtolength{\parenheight}{-.5ex}
	\addtolength{\parenheight}{\parendepth}
	\addtolength{\parendepth}{-.5\parenheight}
	\setlength{\parendrop}{-.5\parenheight}
	\addtolength{\parendrop}{.5ex}
	\raisebox{-\parendepth}{\(#4
		\left#1%
		\rule[\parendrop]{0pt}{\parenheight}%
		\right.\)}
	#2
	\raisebox{-\parendepth}{\(#4
		\left.%
		\rule[\parendrop]{0pt}{\parenheight}%
		\right#3\)}
}

\def\myleft#1#2\myright#3{%
	\mathchoice{%
		\paren{#1}{#2}{#3}{\displaystyle}%
	}{%
	\paren{#1}{#2}{#3}{\textstyle}%
}{%
\paren{#1}{#2}{#3}{\scriptstyle}%
}{%
\paren{#1}{#2}{#3}{\scriptscriptstyle}%
}%
}

\allowdisplaybreaks

\renewenvironment{proof}[1][Proof]{\begin{trivlist} \item[\hskip \labelsep {\bfseries #1:}]}{\qed\end{trivlist}}

\thanks{The research of the first author is supported by the Alfried Krupp Prize for Young University Teachers of the Krupp foundation. The third author was partially supported by the NSF grant DMS-1601070}

\title{Rank two false theta functions and Jacobi forms of negative definite matrix index}

\author[K. Bringmann]{Kathrin Bringmann}
\author[J. Kaszian]{Jonas Kaszian}
\author[A. Milas]{Antun Milas}
\author[S. Zwegers]{Sander Zwegers}

\address{Kathrin Bringmann, University of Cologne, Department of Mathematics and Computer Science, Weyertal 86-90, 50931 Cologne, Germany}
\email{kbringma@math.uni-koeln.de}

\address{Jonas Kaszian, University of Cologne, Department of Mathematics and Computer Science, Weyertal 86-90, 50931 Cologne, Germany}
\email{jkaszian@math.uni-koeln.de}

\address{Antun Milas, Department of Mathematics and Statistics, SUNY-Albany, Albany, NY 12222, U.S.A.}
\email{amilas@albany.edu}

\address{Sander Zwegers, University of Cologne, Department of Mathematics and Computer Science, Weyertal 86-90, 50931 Cologne, Germany}
\email{szwegers@math.uni-koeln.de}

\begin{document}
\begin{abstract}
In this paper, we study a family of rank two false theta series associated to the root lattice of type $A_2$. We show that these functions appear as Fourier coefficients of a meromorphic Jacobi form
of negative definite matrix index. Hypergeometric $q$-series identities are also obtained.
\end{abstract}
\maketitle

\section{Introduction and statement of results}


(Holomorphic) Jacobi forms are complex functions on $\mathbb{C} \times \mathbb{H}$ that transform nicely under the Jacobi group. They generalize modular forms, have an associated weight and index (which is a positive half-integer), and include the classical Jacobi theta function. Jacobi forms are related to modular forms in various ways. One of these connections is through the so-called theta decomposition, which in particular implies that Fourier coefficients (in the $z$-variable) of holomorphic Jacobi forms are modular forms \cite{EZ}. The situation is more difficult in the case that the Jacobi form has poles in the elliptic $z$-variable. In this case more complicated modular type objects occur \cite{BF,DMZ,Ol,Zw}. If the Jacobi forms have poles in $z$, then the index of the Jacobi form may also be negative and in this case false theta functions occur \cite{BCR, BRZ}. False theta functions are similar to theta functions but some of the signs are different, which prevents them from being modular forms. The notion of Jacobi form can be easily generalized to several complex variables. Prominent examples of  multi-variable Jacobi forms come from characters of integrable highest weight modules for affine Lie algebras \cite{KP}. More recently, the study of mock theta functions \cite{Zw}, quantum black holes \cite{DMZ}, and Kac--Wakimoto characters of affine Lie superalgebras \cite{KW} put more emphasis on meromorphic Jacobi forms and their Fourier coefficients.

In \cite{BKM,BM}, the first three authors studied a family of rank two theta-like series (throughout $p\in\N_{\geq 2}$): 
\begin{align*} 
F(\zeta_1,\zeta_2;q)&:=\sum_{n_1,n_2\in\Z}\frac{q^{p\left(\left(n_1-\frac{1}{p}\right)^2+\left(n_2-\frac{1}{p}\right)^2-\left(n_1-\frac{1}{p}\right)\left(n_2-\frac{1}{p}\right)\right)}}{\left(1-\zeta_1^{-1}\right)\left(1-\zeta_2^{-1}\right)\left(1-\zeta_1^{-1}\zeta_2^{-1}\right)}\left(\zeta_1^{n_1-1} \zeta_2^{n_2-1} \!\!-\zeta_1^{-n_1+n_2-1}\zeta_2^{n_2-1} \right.\\
&\qquad\quad-\zeta_1^{n_1-1}\zeta_2^{-n_2+n_1-1}\!+\left.\zeta_1^{-n_2-1} \zeta_2^{-n_2+n_1-1}\!+\zeta_1^{-n_1+n_2-1} \zeta_2^{-n_1-1}\!-\zeta_1^{-n_2-1}\zeta_2^{-n_1-1}\right).
\end{align*}
It is not difficult to see that the rational function in $\zeta_1$ and $\zeta_2$ showing up as the summand in $F(\zeta_1,\zeta_2;q)$ is actually a Laurent polynomial. 
As demonstrated in \cite[Section 4.3]{BM}, the constant term of $F(\zeta_1,\zeta_2;q)$, taken with respect to $\zeta_1$ and $\zeta_2$, is given by 
\begin{equation*}
F(q):=\sum_{n_1,n_2 \geq 1 \atop n_1 \equiv n_2 \ \pmod{3}} {\rm min}(n_1,n_2)\,q^{\frac{p}{3}\left(n_1^2+n_2^2+n_1n_2\right)-n_1-n_2+\frac{1}{p}} \left(1-q^{n_1})(1-q^{n_2})(1-q^{n_1+n_2}\right).
\end{equation*}
In \cite{BKM}, the last expression was called the {\it $\frak{sl}_3$  false theta} function. 
The $q$-series $F$ has several interesting properties, and in particular, it can be used to construct a non-trivial example of a so-called depth two quantum modular form introduced in \cite[Theorem 1.1]{BKM}.
The above construction can be viewed as a two variable version of a more familiar example involving one complex variable $\zeta$ \cite[Section 4.2]{BM}. Indeed, taking the Fourier coefficients of 
\begin{equation} \label{rank-one}
\sum_{n \in \mathbb{Z}}  q^{p\left(n+\frac{p-1}{2p}\right)^2} \frac{\zeta^{2n+1}-\zeta^{-2n-1}}{\zeta-\zeta^{-1}},
\end{equation}
results in classical false theta functions studied by numerous authors. For instance, for $p=2$, the constant coefficient of (\ref{rank-one})  is essentially the classical
Roger's false theta function
$$
\sum_{n \geq 0} (-1)^n q^{\frac{n(n+1)}{2}}.
$$
Interestingly, all Fourier coefficients of (\ref{rank-one}) are basically so-called quantum modular forms with quantum set $\mathbb{Q}$ (see \cite[Theorem 4.1]{BM} for details).

Both (\ref{rank-one}) and $F(\zeta_1,\zeta_2;q)$ are directly related to characters of representations of certain $W$-algebras. As demonstrated by Feigin and Tipunin \cite{FT}, such vertex algebras can be associated to any simply-laced root lattice of ADE type. We omit discussing the precise connection here because it is out of the scope of this paper (see  \cite{AM, BM,C, CM, FT} for more details).

In a somewhat different direction, Kac and Wakimoto \cite[Example 2]{KW} recently obtained product formulas for irreducible characters of  representations of affine Lie algebras at the boundary admissible levels. For the Lie algebra $\widehat{\frak{sl}}_N$ at level $-\frac{N}{2}$, where $N$ is odd, their formula takes an elegant  shape, namely \begin{equation} \label{new}
{\rm ch}\left[L\left(-\tfrac{N}{2} \Lambda_0\right)\right] \left(z_1,z_2,...,z_{N-1}; \tau \right)=\left( \frac{\eta(\tau)}{\eta(2 \tau)} \right)^{\frac12(N-1)(N-2)}  \prod_{1 \leq j \leq k \leq N-1}  \frac{\vartheta\left(\sum_{r=j}^k z_j; 2 \tau\right)}{\vartheta\left(\sum_{r=j}^k z_j;  \tau\right)},
\end{equation}
where throughout $q:=e^{2\pi i \tau}$, $\eta(\tau):=q^{\frac{1}{24}}\prod_{n=1}^\infty(1-q^n)$ is \textit{Dedekind's $\eta$-function}, $\vartheta(z;\tau):=\sum_{n\in\frac12+\mathbb Z}q^{\frac{n^2}{2}}e^{2\pi in(z+\frac12)}$ is the \textit{Jacobi theta function}, and $L(-\frac{N}{2} \Lambda_0)$ denotes the simple affine vertex operator algebra of level $-\frac{N}{2}$. Due to the presence of theta functions, the modular properties of these characters can easily be determined. A much more interesting and harder problem is to investigate modular properties of characters of the corresponding parafermionic vertex algebra $K(\frak{sl}_N,-\frac{N}{2})$ and its representations. Characters of such modules are obtained by taking the Fourier expansion of  ${\rm ch}[L(-\frac{N}{2} \Lambda_0)]$ with respect to the variables $z_j$ in a particular range. Since the Jacobi form appearing on the right-hand side of (\ref{new}) is meromorphic of negative definite (matrix) index it is not clear what modular transformation properties these coefficients should possess. This is in contrast to the situation of the characters of representations of integrable highest weight modules, which are holomorphic and whose Fourier coefficients are modular forms \cite[Section 4.4]{KP}.

The aim of this paper is to connect the Fourier coefficients of $F(\zeta_1,\zeta_2;q)$ for $p=2$ to the Fourier coefficients in Kac--Wakimoto's character formula for $N=3$. Strong hints that they are related come from the work on logarithmic $W$-algebras \cite{FT} combined with work of Adamovic \cite{Ad}, specifically his important realization of $K\left(\frak{sl}_3,-\frac{3}{2} \right)$ \cite[Theorem 11.1]{Ad}. As for $N >3$, we do not know whether Fourier coefficients of (\ref{new}) can be expressed using higher rank false theta functions introduced in \cite[formula (1.2)]{BM}. 

We obtain the following result (see Theorem \ref{TheoremCoeffJacobiExplicit} for a more explicit version) in the case that $p=2$.

\begin{theorem}\label{TheoremCoeffJacobi}
	For $p=2$, all Fourier coefficients of $F(\zeta_1,\zeta_2;q)$ appear as coefficients of
	a (single) meromorphic Jacobi form of matrix index 
	 $-\frac12\left(\begin{smallmatrix} 2& 1 \\ 1 & 2 \end{smallmatrix}\right)$ and weight two with respect to some congruence subgroup. 
\end{theorem}

Motivated by several known $q$-hypergeometric expressions for partial and false theta functions \cite{AW, Wa1,Wa2} (some going back to Ramanujan \cite[entry 9]{Berndt}), using 
results from the proof of Theorem \ref{TheoremCoeffJacobi}, we give $q$-hypergeometric formulas for a class of false theta functions parametrized by pairs of integers $(r_1,r_2)$. In the special case corresponding to $(r_1,r_2)=(0,0)$, we obtain an elegant identity; more general identities are given in Proposition \ref{prop}.
\begin{theorem}\label{thm:elegant identity}
	We have
	\begin{multline*} 
	(q;q)_\infty^{-2} \left(q^2;q^2\right)_\infty^{-2}\sum_{\substack{ n_1 \geq 0 \\ n_2 \in \mathbb{Z}}} {\rm sgn}^*(n_2)(-1)^{n_1}  q^{\frac{n_1(n_1+1)}{2}+n_1n_2+2n_2^2+2n_2 }    \\ 
	= \sum_{\substack{n_1,n_2,n_3\in \N_0\\ n_4\in\Z}} \frac{q^{2n_1+2n_2+2n_3+ 3 |n_4| }}{(q^2;q^2)_{n_1} (q^2;q^2)_{n_1+|n_4|}(q^2;q^2)_{n_2} (q^2;q^2)_{n_2+ |n_4|}  (q^2;q^2)_{n_3} (q^2;q^2)_{n_3+ |n_4|}},
	\end{multline*}
	where $(a;q)_n:=\prod_{j=0}^{n-1}(1-aq^j)$ for $n\in\N_0\cup\{\infty\}$ and where ${\rm sgn}^*(n):=1$ if $n \geq 0$ and $-1$ otherwise.
\end{theorem}

The paper is organized as follows. In Section \ref{sec-JacobiTheta}, we discuss the classical Jacobi theta function and certain Jacobi forms in two variables. In Section \ref{sec-proofSecion}, we prove Theorem \ref{TheoremCoeffJacobi} in several steps: In Subsection \ref{sec-coeffF} we introduce a family of false theta functions in two summation variables, denoted by $\mathbb{G}_{(\lambda_1,\lambda_2)}$. 
These $q$-series are essentially generalizations of the rank two false theta function $F(q)$. 
 In Proposition \ref{ct} we give an explicit formula for the Fourier coefficients of $F(\zeta_1,\zeta_2;q)$ in terms of $\mathbb G_{(\lambda_1,\lambda_2)}$. We next compute the Fourier coefficients of the relevant Jacobi form in Subsection \ref{sec-JacobiCoeffs}.
Finally we combine these results to prove the first main result of the paper in Subsection \ref{sec-proofMainTheorem}. In Section \ref{sec-additionalChars}, we determine the Fourier coefficients of two additional characters from \cite{KW} using results from Subsection \ref{sec-JacobiCoeffs}. In Section 5, we employ results from Section 4 to give $q$-hypergeometric formulas for a class of functions $\mathbb{G}_{(\lambda_1,\lambda_2)}$.
Section 6 is concerned with the $q$-series $\lim_{(\zeta_1,\,\zeta_2) \to (1,1)} F(\zeta_1,\zeta_2;q)$, for $p=2$. This limit is important from the point of view of representation theory \cite[Section 4.3]{BM}.
We show that it can be computed as the constant term of a particular index zero Jacobi form (see Theorem \ref{index-zero}).

\section*{Acknowledgements}
The authors thank Chris Jennings-Shaffer for helpful comments on an earlier version of this paper.

\section{Jacobi theta functions}\label{sec-JacobiTheta}
We first recall some properties of the Jacobi theta function. By the Jacobi triple product identity, we have 
\[
\vartheta(z;\tau)= -iq^{\frac18} \zeta^{-\frac12} \left(\zeta,\zeta^{-1}q,q;q\right)_\infty,
\]
where $(a_1,\ldots,a_\ell;q)_n:=(a_1;q)_n\ldots(a_\ell;q)_n$ and $\zeta:=e^{2\pi i z}$ throughout. Then, with $\chi$ the multiplier of $\eta$, we have for $m,\ell\in\Z$ and $\gamma=\left(\begin{smallmatrix}
a & b \\ c & d
\end{smallmatrix}\right)\in\SL_2(\Z)$
\begin{align}\label{TranE}
\vartheta(z+m\tau+\ell;\tau)&=(-1)^{m+\ell}q^{-\frac{m^2}{2}}\zeta^{-m}\vartheta(z;\tau),\\
\vartheta\left(\frac{z}{c\tau+d};\frac{a\tau+b}{c\tau+d}\right)&=\chi(\gamma)^3(c\tau+d)^\frac12e^{\frac{\pi icz^2}{c\tau+d}}\vartheta(z;\tau).\label{TranM}
\end{align}

In this paper, we deal with functions satisfying a higher-dimensional generalization of these transformations which we now recall. Here and throughout bold letters denote vectors such as $\b{z}\in\C^N$, $N\in\N$.
\begin{definition}
	Let $L_1,L_2\subset \Z^N$ be lattices, $\nu_1:\Gamma\to S^1:= \{z\in\C: \lvert z\rvert=1 \}$ a multiplier, and $\nu_2: L_1\times L_2:\Gamma\to S^1$ a homomorphism with finite image, and $N\in\N$.
	We call a meromorphic function $g:\C^N\times \H\to \C$ a \emph{Jacobi form of matrix index} $M\in \nolinebreak\frac14\Z^{N\times N}$ (with $M^T=M$ and  $M_{j,j}\in\frac12\Z$ for $j\in\{1,\dots,N\}$) and weight $k\in\frac12\Z$ for $\Gamma\subset\SL_2(\Z)$ with respect to $L_1\times L_2$ and $\nu_1$,$\nu_2$ if it satisfies the following transformation laws (for all $(\b{z}, \tau) \in \C^N\times \H$): 
	\begin{enumerate}
		\item For $\b{m}\in L_1$, $\b{\ell}\in L_2$ we have
		\begin{align*}
		g\left(\b{z}+\b{m}\tau +\b{\ell};\tau \right) = \nu_2(\b{m},\b{\ell}) q^{-\b{m}^T M \b{m}} e^{-4\pi i \b{m}^T M \b{z}}
		g(\b{z};\tau).
		\end{align*}
		
		\item For $\gamma=\left(\begin{smallmatrix}
		a & b\\c & d
		\end{smallmatrix}\right)\in \Gamma$ we have
		\begin{align*}
		g\left(\frac{\b{z}}{c\tau+d};\frac{a\tau+b}{c\tau+d}\right) = \nu_1(\gamma) \left(c\tau+d\right)^k e^{\frac{2\pi i c}{c\tau +d} \b{z}^T M \b{z}}
		g(\b{z};\tau).
		\end{align*}
		\item For some $a>0$, we have
		\begin{align*}
		g(\b{z};\tau) e^{-\frac{4\pi }{\im(\tau)} \im(\b{z})^TM\im(\b{z})}\in O\left(e^{a\im(\tau)}\right)     \qquad \text{as }\Im(\tau)\to \infty.
		\end{align*}.   
	\end{enumerate}
We say that $g$ is of positive (resp. negative) matrix index if $M$ is a positive (resp. negative) definite matrix.
\end{definition}

In this paper, we are concerned with the Jacobi form
\begin{equation} \label{f}
f(\b z;\tau):=\frac{\vartheta(z_1;2\tau)\vartheta(z_2;2\tau)\vartheta(z_1+z_2;2\tau)}{\vartheta(z_1;\tau)\vartheta(z_2;\tau)\vartheta(z_1+z_2;\tau)}
\end{equation}
and a slightly dilated $A_2$ theta function
\[
\mathcal{T}(\b z;\tau):=\Theta_{A_2}(z_1+2z_2,z_1-z_2;2\tau),\qquad \textnormal{where }\quad\Theta_{A_2}(\b z;\tau):=\sum_{\b n \in \mathbb{Z}^2} q^{Q(\b n)} e^{2\pi i(n_1z_1+n_2z_2)}
\]
with the quadratic form $Q(\b n):=n_1^2+n_2^2-n_1n_2$. We prove the following transformation.
\begin{proposition} \label{prop-index} The function $f$ is Jacobi form of weight zero and matrix index $-\frac12\left(\begin{smallmatrix} 2& 1 \\ 1 & 2 \end{smallmatrix}\right)$. 
 More precisely, we have for $\gamma=\left(\begin{smallmatrix}
a & b \\ c & d
\end{smallmatrix}\right)\in\Gamma_0(2)$, $\b m\in2\Z^2$, and $\b \ell\in\Z^2$,  
\begin{align*}
f \left(\frac{\b z}{c\tau+d};\frac{a \tau+b}{c \tau+d}\right)&= \nu(\gamma) e^{-\frac{\pi i c}{c \tau+d}Q^*(\b{z})} f(\b z;\tau),  \\
f(\boldsymbol{z}+ \boldsymbol{m} \tau + \boldsymbol{\ell};\tau)&= q^{\frac12 Q^*(\b{m})} \zeta_1^{m_1 + \frac{m_2}{2}} \zeta_2 ^{m_2 + \frac{m_1}{2}} f(\boldsymbol{z};\tau),
\end{align*}
where  $\nu(\gamma) :=\chi\left(\begin{smallmatrix}
a & 2b \\ \frac c2 & d
\end{smallmatrix}\right)^9\chi(\gamma)^{-9}$, $Q^*(\b{z}):=z_1^2+z_2^2+z_1 z_2$, and $\zeta_j:=e^{2\pi iz_j}$.

The theta function $\mathcal{T}$ is a weight one Jacobi form of matrix index $\frac12\left(\begin{smallmatrix} 2& 1 \\ 1 & 2 \end{smallmatrix}\right)$.
To be more precise, we have for $\left(\begin{smallmatrix}
a&b\\ c&d
\end{smallmatrix}\right) \in \Gamma_0(6)$, $\boldsymbol{m} \in \Z^2$, and $\boldsymbol{\ell} \in \Z^2$ 
\begin{align}\label{Tmod}
\mathcal{T}\left(\frac{\boldsymbol{z}}{c\tau +d}; \frac{a\tau +b}{c\tau +d}\right) &= \left(\frac{-3}{d}\right) (c\tau +d) e^{\frac{\pi i c}{c\tau +d}Q^*(\b{z})} \mathcal{T}(\boldsymbol{z};\tau),\\ \label{Tel}
\mathcal{T}(\boldsymbol{z}+ \boldsymbol{m}\tau + \boldsymbol{\ell}; \tau)&= q^{-\frac12 Q^*(\b{m})} \zeta_1^{-m_1- \frac{m_2}{2}} \zeta_2^{-\frac{m_1}{2}- m_2} \mathcal{T}(\boldsymbol{z};\tau),
\end{align}
where $\left(\frac{\ { \cdot} \  }{\cdot}\right)$ denotes the Jacobi symbol.
\end{proposition}
\begin{proof}
Using \eqref{TranE} and \eqref{TranM}, we obtain the claims for $f$. 

To prove \eqref{Tmod}, we use Proposition 3.8 of \cite{KP}, which gives for $A=\left(\begin{smallmatrix}
a& b \\ c & d
\end{smallmatrix}\right)\in\Gamma_0(3)$, that

\begin{align*}
\Theta_{A_2}\left(\frac{\b{z}}{c\tau+d};\frac{a\tau+b}{c\tau+d}\right) =\left(\frac{-3}{d}\right)(c\tau+d) e^{\frac{2\pi i c}{3(c\tau+d)} Q^*(\b{z})}
\Theta_{A_2}(\b z;\tau).
\end{align*}
The elliptic transformation \eqref{Tel} follows from 
$$\Theta_{A_2}(\boldsymbol{z}+ \boldsymbol{m}\tau + \boldsymbol{\ell};\tau)=  q^{- \frac13 Q^*(\b{m})} \zeta_1^{-\frac12 (m_2+2m_1)} \zeta_2^{-\frac12 (m_1 +2 m_2)}\Theta_{A_2}(\boldsymbol{z};\tau),$$ which can be confirmed by a direct calculation.
\end{proof}


\section{Proof of Theorem \ref{TheoremCoeffJacobi}}\label{sec-proofSecion}
In this section, we prove an explicit version of the main result of this paper, Theorem \ref{TheoremCoeffJacobi}, using the Jacobi form $f$ defined in \eqref{f} in the case $p=2$.
\begin{theorem}\label{TheoremCoeffJacobiExplicit}
	For $p=2$, $\boldsymbol{r}\in\Z^2$, and $\lvert q\rvert<\lvert \zeta_j\lvert  <1, \lvert q\rvert<\lvert \zeta_1\zeta_2\lvert  <1$, $j \in \{1,2 \}$, we have
	\begin{align*}
	{\rm coeff}_{\left[\zeta_1^{r_1},\,\zeta_{2}^{r_2}\right]} F\left(\zeta_1,\zeta_2;q\right)=q^{2Q(\b r)} \frac{\eta(\tau)^5}{\eta(2\tau)}
	{\rm coeff}_{\left[\zeta_1^{2r_1-r_2},\,\zeta_{2}^{r_1+r_2}\right]} f\left(\b{z};\tau\right),
	\end{align*}
	where ${\rm coeff}_{[\zeta_1^{r_1},\,\zeta_2^{r_2}]}$ denotes the $r_1$-th Fourier coefficient in $z_1$ and the $r_2$-th Fourier coefficient in $z_2$.
\end{theorem}
The proof consists of several steps. We first determine an explicit expression for the coefficients of $F$ and then compare them with the coefficients of $f$.

\subsection{Coefficients of $F(\zeta_1,\zeta_2;q)$}\label{sec-coeffF}

In this section we compute the Fourier coefficients of $F(\zeta_1,\zeta_2;q)$ for general $p$. For this we require, for ${\b\lambda}\in \mathbb{Q}^2$,
\begin{multline*}
\mathbb{G}_{\boldsymbol{\lambda}}(\tau):=\sum_{\b{n} \in \N^2} {\rm min}(n_1,n_2)\,q^{p Q\left(\b n+\boldsymbol{\lambda}- \left(\frac1p, \frac1p\right)\right)}   \\
\times \left(1-q^{2\left(n_1+\lambda_1\right)-\left(n_2+\lambda_2\right)}-q^{2\left(n_2+\lambda_2\right)-\left(n_1+\lambda_1\right)}+q^{3\left(n_1+\lambda_1\right)}+q^{3\left(n_2+\lambda_2\right)}-q^{2\left(n_1+\lambda_1\right)+2\left(n_2+\lambda_2\right)}\right).\notag
\end{multline*}
The following result resembles  \cite[Proposition 5.1]{BM}.
\begin{proposition} \label{ct}
 For all $\b r \in \mathbb{Z}^2$, we have 
\[
{\rm coeff}_{\left[\zeta_1^{r_1},\,\zeta_{2}^{r_2}\right]} F(\zeta_1,\zeta_2;q)=\mathbb{G}_{\boldsymbol{r}}(\tau).
\]
\end{proposition}
\begin{proof}
Denoting the summands (without the minus-factors) appearing in the definition of $F(\zeta_1,\zeta_2;q)$ by $F_j(\zeta_1,\zeta_2;q)$ and the summands 
of $\mathbb{G}_{\b r}(\tau)$ by $\mathbb{G}_{\b r,j}(\tau)$. We claim that
\begin{equation}\label{claim}
{\rm coeff}_{\left[\zeta_1^{r_1},\,\zeta_2^{r_2}\right]} F_j(\zeta_1,\zeta_2;q)=\mathbb{G}_{\b r,j}(\tau).
\end{equation}
We prove \eqref{claim} only for  $j\in\{1,2\}$, the other cases are shown analogously. 
In fact, we only show \eqref{claim} for $|\zeta_k|>1$ ($k\in\{1,2\}$).
Since the rational function in $\zeta_1$ and $\zeta_2$ showing up as the summand in $F(\zeta_1,\zeta_2;q)$ is actually a Laurent polynomial, the statement of Proposition\ \ref{ct} then holds for all $\zeta_k$.
For this we expand the Weyl denominator in $F(\zeta_1,\zeta_2;q)$ in non-positive powers of $\zeta_1$ and $\zeta_2$
\begin{equation} \label{weyl-sl3}
\frac{1}{\left(1-\zeta_1^{-1}\right)\left(1-\zeta_2^{-1}\right)\left(1-\zeta_1^{-1}\zeta_2^{-1}\right)}=\sum_{\b \ell\in\N_0^2} {\rm min}(\ell_1+1,\ell_2+1) \zeta_1^{-\ell_1} \zeta_2^{-\ell_2}.
\end{equation} 

We start with $j=1$. We have, using \eqref{weyl-sl3},  
\begin{align*}
&{\rm coeff}_{\left[\zeta_1^{r_1},\,\zeta_2^{r_2}\right]} F_1(\zeta_1,\zeta_2;q) \\
&= {\rm coeff}_{\left[\zeta_1^{r_1},\,\zeta_2^{r_2}\right]} \sum_{\b \ell \in\N_0^2} {\rm min}(\ell_1+1,\ell_2+1) \sum_{\b n \in \Z^2} {q^{{pQ\left(\b n-\left(\frac1p,\frac1p\right) \right)}}}  \zeta_1^{n_1-\ell_1-1} \zeta_2^{n_2-\ell_2-1}\\
& = \sum_{\b \ell \in \N_0^2} {\rm min}(\ell_1+1,\ell_2+1)  {q^{ pQ\left(\b \ell+\b r+\left(1-\frac1p,1-\frac1p\right)\right) }} = \sum_{\b \ell \in \N^2} {\rm min}(\ell_1,\ell_2)  {q^{pQ\left(\b \ell+\b r-\left(\frac1p,\frac1p\right)\right)}}=\mathbb{G}_{\boldsymbol{r},1}(\tau), 
\end{align*}
shifting $\b \ell\mapsto\b \ell - (1,1)$ for the second to last equality.

For $j=2$, we have 
\begin{align*}
&{\rm coeff}_{\left[\zeta_1^{r_1},\,\zeta_2^{r_2}\right]} F_2(\zeta_1,\zeta_2;q)={\rm coeff}_{\left[\zeta_1^{r_1},\,\zeta_2^{r_2}\right]} \sum_{\b n \in \Z^2} \frac{q^{pQ\left(\b n-\left(\frac1p,\frac1p\right)\right)}}{\left(1-\zeta_1^{-1}\right)\left(1-\zeta_2^{-1}\right)\left(1-\zeta_1^{-1}\zeta_2^{-1}\right)} \zeta_1^{-n_1+n_2-1} \zeta_2^{n_2-1} \\
& = {\rm coeff}_{\left[\zeta_1^{r_1},\,\zeta_2^{r_2}\right]} \sum_{\b \ell \in\N_0^2} {\rm min}(\ell_1+1,\ell_2+1) \sum_{\b n \in \Z^2} {q^{pQ\left(\b n-\left(\frac1p,\frac1p\right)\right)}}  \zeta_1^{-n_1+n_2-r_1-1} \zeta_2^{n_2-r_2-1} \\
&= {\rm coeff}_{\left[\zeta_1^{r_1},\, \zeta_2^{r_2}\right]} \sum_{\b \ell \in\N_0^2} {\rm min}(\ell_1+1,\ell_2+1)\sum_{\b n \in \Z^2} {q^{pQ\left(n_2-n_1-\frac1p,\,n_2-\frac1p\right)}}  \zeta_1^{n_1-\ell_1-1} \zeta_2^{n_2-\ell_2-1} \\
&= \sum_{\b \ell \in \N_0^2} {\rm min}(\ell_1+1,\ell_2+1)  {q^{pQ\left(\ell_2-\ell_1+r_2-r_1-\frac1p,\,\ell_2+r_2+1-\frac1p\right)}}=\mathbb{G}_{\boldsymbol{r},2}(\tau),
\end{align*}
where we change $n_1\mapsto -n_1+n_2$ for the third equality and
\begin{multline*}
pQ\left(\ell_2-\ell_1+r_2-r_1-\frac1p,\ell_2+r_2+1-\frac1p\right) \\
= 2(\ell_1+1+r_1)-(\ell_2+1+r_2)+pQ\left(\b \ell+\b r + \left(1-\frac1p,1-\frac1p\right)\right)
\end{multline*}
for the final equality.
\end{proof}

\begin{remark}
The first sum  appearing in the definition of $\mathbb{G}_{\mathbb{\boldsymbol{\lambda}}}$, namely 
$$\sum_{\b{n} \in \N_0^2} {\rm min}(n_1,n_2)q^{pQ\left(\b n+ \b \lambda - \left(\frac1p,\frac1p\right)\right)},$$
is an example of Kostant's partial theta function of type $A_2$ (see \cite[Section 3]{CM}). 
\end{remark}

\subsection{Partial theta functions as Fourier coefficients of Jacobi forms}\label{sec-JacobiCoeffs}

From now on until the end of this section, we assume that $p=2$ and
write $\mathbb G_{\boldsymbol{\lambda}}$ as Fourier coefficients of the Jacobi form $f$. 
\begin{proposition} \label{ConjG}
For $p=2$ and for all $\b r \in \mathbb{Z}^2$ and $|q| < |\zeta_j|<|q|^{-1}$, $|q|<|\zeta_1 \zeta_2 | < |q|^{-1}$, $j \in \{1,2 \}$, we have 
$$q^{-\frac23Q(\b r)} \mathbb{G}_{\frac13(r_1+r_2,\,2r_2-r_1)}(\tau) =   \frac{\eta(\tau)^5}{\eta(2 \tau)}\,\textnormal{coeff}_{\left[\zeta_1^{r_1},\,\zeta_2^{r_2}\right]}  f(\b z;\tau).$$

\end{proposition}

To prove Proposition \ref{ConjG}, we first rewrite $\mathbb G_{\boldsymbol{\lambda}}$.
\begin{lemma}\label{la:rewriteG}
	We have, for $p=2$ and $\b\lambda\in\Q^2$,
	\begin{align*}
	\mathbb{G}_{\b{\lambda}}(\tau)=
	\sum_{\b n\in\N_0^2}  q^{2Q\left(\b{n}+\b{\lambda}+\left(\frac12,\frac12\right)\right)}
	-\sum_{n_2>n_1\geq 0} q^{2Q\left(\b{n}+\b{\lambda}+\left(\frac12,\, 0\right)\right)}
	-\sum_{n_1>n_2\geq 0} q^{2Q\left(\b{n}+\b{\lambda}+\left(0,\frac12\right)\right)}.
	\end{align*}
\end{lemma}

\begin{proof}
	It is not hard to see that
	\begin{align}
	\mathbb{G}_{\boldsymbol{\lambda}}(\tau)=
	\sum_{\b{\alpha}\in \mathscr{T}} \kappa(\b{\alpha})
	\sum_{\b n\in\N^2} {\rm min}(n_1,n_2)\,q^{2Q\left(\b{n}+\b{\lambda}+\b{\alpha}\right)},
	\label{Gshape1}
	\end{align}
	with 
	\begin{align*}
	\mathscr{T}&:= \left\{  \left(-\tfrac12,-\tfrac12 \right),\left(0,-\tfrac12 \right)
	,\left(-\tfrac12,0 \right),\left(\tfrac12,0 \right),\left(0,\tfrac12 \right),\left(\tfrac12,\tfrac12 \right)\right\}\\
	\kappa(\b{\alpha})&:=\begin{cases}
	1 \quad & \text{if }\b{\alpha}\in  \left\{\left(-\frac12,-\frac12 \right),\left(\frac12,0 \right),\left(0,\frac12 \right)\right\},\\
	-1 \quad & \text{if }\b{\alpha}\in \left\{\left(0,-\frac12 \right)
	,\left(-\frac12,0 \right),\left(\frac12,\frac12 \right)\right\}.
	\end{cases}
	\end{align*}
	 We combine terms in \eqref{Gshape1} suitably and shift $\b n\mapsto \b n+(1,1)$ in the $\b{\alpha}=(-\frac12,-\frac12)$ term to obtain
	\begin{equation*}
	\quad \sum_{\b{\alpha}\in \left\{\left(-\frac12,-\frac12 \right),\left(\frac12,\frac12 \right)\right\}} \kappa\left(\b{\alpha}\right)
	\sum_{\b n\in\N^2} \min\left(n_1,n_2\right) q^{2Q\left(\b{n}+\b{\lambda}+\b{\alpha}\right)}=
	\sum_{\b n\in\N_0^2}  q^{2Q\left(\b{n}+\b{\lambda}+\left(\frac12,\frac12\right)\right)}.
	\end{equation*}
	Similarly, we have
	\begin{align*}
	\sum_{\b{\alpha}\in \left\{\left(-\frac12,0 \right),\left(\frac12,0 \right)\right\}} \kappa\left(\b{\alpha}\right)
	\sum_{\b n\in\N^2} \min\left(n_1,n_2\right) q^{2Q\left(\b{n}+\b{\lambda}+\b{\alpha}\right)}&=
	-\sum_{n_2>n_1\geq 0} q^{2Q\left(\b{n}+\b{\lambda}+\left(\frac12,0\right)\right)},\\
	 \sum_{\b{\alpha}\in \left\{\left(0,-\frac12 \right),\left(0,\frac12 \right)\right\}} \kappa\left(\b{\alpha}\right)
	\sum_{\b n\in\N^2} \min\left(n_1,n_2\right) q^{2Q\left(\b{n}+\b{\lambda}+\b{\alpha}\right)}
	&=
	-\sum_{n_1>n_2\geq 0} q^{2Q\left(\b{n}+\b{\lambda}+\left(0,\frac12\right)\right)}.
	\end{align*}
	This finishes the proof of Lemma \ref{la:rewriteG}.
\end{proof}

We next rewrite the right-hand side of Proposition \ref{ConjG}. For this, we let
\begin{equation*}
\varrho_{n_1,n_2}:=\frac12(\sgn^*(n_1)+\sgn^*(n_2)).
\end{equation*}
\begin{lemma}\label{JacobiExplicit}
For $\b{r} \in \mathbb{Z}^2$ we have,	where $\zeta_j$ $(j\in\{1,2\})$ satisfy $\lvert q\rvert<\lvert \zeta_j\lvert  <1, \lvert q\rvert<\lvert \zeta_1\zeta_2\lvert  <1$ 
	\begin{align*}
	\frac{\eta(\tau)^5}{\eta(2 \tau)}\, {\rm coeff}_{\left[\zeta_1^{r_1},\, \zeta_2^{r_2}\right]}f(\b z;\tau) 
	=\sum_{\substack{n_1\geq 0 \\ n_2\in \Z}} \varrho_{n_2,n_2+r_2} (-1)^{n_1} q^{\frac{n_1(n_1+1)}{2}+n_1n_2+2n_2^2+r_1n_1+2r_2n_2+2n_2+r_2+\frac12}.
	\end{align*}
\end{lemma}
\begin{proof}
Using the product expansion of $\vartheta$ we can easily verify that
\begin{align*}
\frac{\vartheta(z;2\tau)}{\vartheta(z;\tau)}=-i\zeta^{-\frac12} q^{-\frac14} \frac{\eta(2\tau)^2}{\eta(\tau)\vartheta(z+\tau;2\tau)}.
\end{align*}
Using this identity twice we find that the left-hand side of Lemma \ref{JacobiExplicit} equals 
\begin{equation*}
\ell(\tau):=q^{\frac12}
{\rm coeff}_{\left[\zeta_1^{r_1},\,\zeta_2^{r_2}\right]}
\frac{-i\zeta_1^{-\frac12}\eta(\tau)^3}{
	\vartheta(z_1; \tau)}\, h(\b{z};\tau),
\end{equation*}
where 
\begin{align*}
h(\b{z};\tau):=-\frac{i\eta( 2\tau)^3\zeta_2^{-1}q^{-1}\vartheta(z_1; 2\tau)}{
	\vartheta\left(z_2+\tau; 2\tau\right) \vartheta\left(z_1+z_2+\tau;2\tau\right)}.
\end{align*}

We next rewrite $h(\b z;\tau)$.
For this we recall an identity going back to  Jordan and Kronecker (which can be concluded from Theorem 3 of \cite{Za}), which holds for $\lvert q\rvert <\lvert \zeta_1 \rvert <1$,
\begin{equation}\label{qid}
\frac{-i\eta(\tau)^3\vartheta(z_1+z_2;\tau)}{\vartheta(z_1;\tau)\vartheta(z_2;\tau)} =\sum_{n\in\Z} \frac{\zeta_1^n}{1-\zeta_2 q^n}= \sum_{\b n\in\Z^2} \varrho_{n_1,n_2}\, q^{n_1n_2} \zeta_1^{n_1} \zeta_2^{n_2}.
\end{equation}
For the last equality, we use the geometric series expansion to find that, for $|q|<|\zeta_2|<1$,
\begin{equation}\label{geometric}
\frac1{1-\zeta_2q^n}= \sum_{n_1\in\Z} \varrho_{n_1,n}\, q^{n_1n}\zeta_2^{n_1}.
\end{equation}

From \eqref{TranE} and \eqref{qid} we have, for $\lvert q\rvert<\lvert \zeta_2\lvert  <1, \lvert q\rvert<\lvert \zeta_1\zeta_2\lvert  <1$,
$$ h(\b{z};\tau)= -\frac{i\eta(2\tau)^3 \vartheta(-z_1+2\tau;2\tau)}{\vartheta(z_2+\tau;2\tau)\vartheta(-z_1-z_2+\tau;2\tau)}= \sum_{\b{n}\in\Z^2}\varrho_{n_1,n_2}\, q^{2n_1n_2+n_1+n_2} \zeta_1^{-n_2}\zeta_2^{n_1-n_2}.$$

We also need the following partial fraction decomposition, which holds for $\lvert q\rvert <\lvert \zeta_1\rvert <1$ (see \cite[equation (2.1)]{An-1984}), using again \eqref{geometric},
\begin{equation*}
-\frac{i\zeta_1^{-\frac12}\eta(\tau)^3}{
	\vartheta(z_1; \tau)}=\sum_{n\in\Z}\frac{(-1)^nq^{\frac{n(n+1)}{2}}}{1-\zeta_1q^n}=\sum_{\boldsymbol{n}\in\Z^2} \varrho_{n_1,n_2} (-1)^{n_1} q^{\frac{n_1(n_1+1)}{2}+n_1n_2}\zeta_1^{n_2}.
\end{equation*}
Thus
\begin{align*}
\ell(\tau) &= q^{\frac12} \operatorname{coeff}_{\left[\zeta_1^{r_1},\,\zeta_2^{r_2}\right]}\sum_{\b{n}\in\Z^4} \varrho_{n_1,n_2}\varrho_{n_3,n_4} (-1)^{n_1} q^{\frac{n_1(n_1+1)}{2}+n_1n_2+2n_3n_4+n_3+n_4}\zeta_1^{n_2-n_4} \zeta_2^{n_3-n_4}\\
&=\sum_{\boldsymbol{n}\in\Z^2} \varrho_{n_1,n_2+r_1} \varrho_{n_2+r_2,n_2} (-1)^{n_1} q^{\frac{n_1(n_1+1)}{2}+n_1n_2+2n_2^2+r_1n_1+2r_2n_2+2n_2+r_2+\frac12}.
\end{align*}
Using that
\begin{equation*}
\sum_{n_1\in\Z} (-1)^{n_1} q^{\frac{n_1(n_1+1)}{2} + (n_2+r_1)n_1} = 0,
\end{equation*}
we may conclude that
\[\sum_{n_1\in\Z} \varrho_{n_1,n_2+r_1} (-1)^{n_1} q^{\frac{n_1(n_1+1)}{2}+\left(n_2+r_1\right)n_1} = \sum_{n_1\geq 0} (-1)^{n_1} q^{\frac{n_1(n_1+1)}{2}+(n_2+r_1)n_1}.\]
This then gives the claim of the lemma.
\end{proof}

We are now ready to prove Proposition \ref{ConjG}.

\begin{proof}[Proof of Proposition \ref{ConjG}]
We use the identity ($\b{m}\in\Q^2$)
\begin{align*}
	Q\left(\b{m}\right) &= Q\left(m_1-m_2,-m_2\right)=Q(-m_1,m_2-m_1),
\end{align*}
and Lemma \ref{la:rewriteG}, to rewrite 
	\begin{align*}
	\mathbb G_{\b \lambda}(\tau)&=
	\sum_{n_1\geq n_2\geq 0}  q^{2Q\left(\b{n}+\b{\lambda}+\left(\frac12,\frac12\right)\right)}
	+\sum_{n_2>n_1\geq 0}  q^{2Q\left(\b{n}+\b{\lambda}+\left(\frac12,\frac12\right)\right)}\\
	&\hspace{5.5cm}
	 -\sum_{n_2>n_1\geq 0} q^{2Q\left(\b{n}+\b{\lambda}+\left(\frac12,\,0\right)\right)}
	-\sum_{n_1>n_2\geq 0} q^{2Q\left(\b{n}+\b{\lambda}+\left(0,\,\frac12\right)\right)}\\
	&
	=\sum_{n_1\geq n_2\geq 0}  q^{
		2Q\left(n_1-n_2+\lambda_1-\lambda_2,-n_2-\lambda_2-\frac12\right)}
	+\sum_{n_2>n_1\geq 0}  q^{2Q\left(-n_1-\lambda_1-\frac12,\,n_2-n_1+\lambda_2-\lambda_1\right)}
	\\&\hspace{1cm}
	-\sum_{n_2>n_1\geq 0} q^{2Q\left(-n_1-\lambda_1-\frac12,\,n_2-n_1+\lambda_2-\lambda_1-\frac12\right)}
	-\sum_{n_1>n_2\geq 0} q^{2Q\left(n_1-n_2+\lambda_1-\lambda_2-\frac12,-n_2-\lambda_2-\frac12\right)}\\
	&=
	\sum_{\b n\in\N_0^2} \left(q^{2Q\left(n_1+\lambda_1-\lambda_2,-n_2-\lambda_2-\frac12\right)}+q^{2Q\left(-n_1-\lambda_1-\frac12 ,\, n_2+\lambda_2-\lambda_1+1\right)}\right.
	\\&\qquad \left. \hspace{3.5cm}-q^{2Q\left(-n_1 -\lambda_1-\frac12, \, n_2+\lambda_2-\lambda_1+\frac12\right)}-q^{2Q\left(n_1+\lambda_1-\lambda_2+\frac12,-n_2-\lambda_2-\frac12\right)}\right)
	\end{align*}
	shifting $n_1\mapsto n_1+n_2$ in the first sum, $n_2\mapsto n_1+n_2+1$ in the second and third sum, and $n_1\mapsto n_1+n_2+1$ in the final term. Combining the first and the last sum, and also the second and the third sum, we get
	\[\mathbb{G}_{\b{\lambda}}(\tau)= \sum_{\b n\in\N_0^2} (-1)^{n_1} q^{Q\left(\frac{n_1}{2}+\lambda_1-\lambda_2,-n_2-\lambda_2-\frac12\right)}-\sum_{\boldsymbol{n}\in \N_0^2} (-1)^{n_2} q^{Q\left(-n_1-\lambda_1-\frac12,\,\frac12(n_2+1)+\lambda_2-\lambda_1\right)}.\]
	We now assume that $\lambda_1+\lambda_2\in\Z$. Substituting  $\boldsymbol{n}\mapsto(-1-\lambda_1-\lambda_2-n_2,-1-n_1)$, we see that the second sum equals
	\[ \sum_{\substack{n_1<0 \\ n_2<-\lambda_1-\lambda_2}} (-1)^{n_1} q^{Q\left(n_2+\lambda_2+\frac12,-\frac{n_1}{2}+\lambda_2-\lambda_1\right)}.\]
	Using $Q(\b{m})=Q(-m_2,-m_1)$ we find that
\begin{equation*}
\mathbb{G}_{\b{\lambda}}(\tau)=\left(\sum_{\b n\in\N_0^2} + \sum_{\substack{n_1<0 \\ n_2 < -\lambda_1-\lambda_2}}\right) (-1)^{n_1} q^{Q\left(\frac{n_1}{2}+\lambda_1-\lambda_2,-n_2-\lambda_2-\frac12\right)}.
\end{equation*}
In particular, using this with  $\lambda_1=\frac13(r_1+r_2)$ and $\lambda_2=\frac13(2r_2-r_1)$ (note that $\lambda_1+\lambda_2=r_2\in\Z$ is satisfied) we obtain that
\begin{align*}
q^{-\frac23Q(\b r)}\, \mathbb{G}_{\frac13(r_1+r_2,2r_2-r_1)}(\tau)&=\left(\sum_{\b n\in\N_0^2 } + \sum_{\substack{n_1 < 0 \\ n_2 < -r_2}}\right) (-1)^{n_1} q^{\frac{n_1(n_1+1)}{2} + n_1n_2+2n_2^2+r_1n_1+2r_2n_2+2n_2+r_2+\frac12}\\
&=\sum_{n_1\geq 0} \left(\sum_{n_2\geq 0} - \sum_{n_2 < -r_2}\right) (-1)^{n_1} q^{\frac{n_1(n_1+1)}{2} + n_1n_2+2n_2^2+r_1n_1+2r_2n_2+2n_2+r_2+\frac12}\\
&= \sum_{\substack{n_1\geq 0 \\ n_2\in\Z }} \varrho_{n_2,n_2+r_2}(-1)^{n_1} q^{\frac{n_1(n_1+1)}{2} + n_1n_2+2n_2^2+r_1n_1+2r_2n_2+2n_2+r_2+\frac12},
\end{align*}
where in the penultimate step we use the same argument as at the end of the proof of Lemma \ref{JacobiExplicit}.
Applying Lemma \ref{JacobiExplicit} gives the claim for $|q| < |\zeta_j|<1$, $|q|<|\zeta_1 \zeta_2 | <1$. Using
 \begin{equation} \label{T2T}
\frac{\vartheta (z;2 \tau)}{\vartheta(z;\tau)}=\frac{q^{\frac18} (-q;q)_\infty }{(\zeta q,\zeta^{-1}q ;q^2)_\infty}
\end{equation}
combined with the uniqueness of the Laurent expansion inside the domain
$|q| < |\zeta_j|<|q|^{-1}$, $|q|<|\zeta_1 \zeta_2 | < |q|^{-1}$
extends the claim to that domain.
\end{proof}

\subsection{Combining the results}\label{sec-proofMainTheorem}
\begin{proof}[Proof of Theorem \ref{TheoremCoeffJacobiExplicit} and Theorem \ref{TheoremCoeffJacobi}]
	We use Proposition \ref{ct} and then Proposition \ref{ConjG} (with $\boldsymbol{r}\mapsto\left(2r_1-r_2,r_1+r_2\right)$) to obtain 
	\begin{align*}
	{\rm coeff}_{\left[\zeta_1^{r_1},\,\zeta_{2}^{r_2}\right]} F\left(\zeta_1,\zeta_2;q\right)=
	\mathbb{G}_{\boldsymbol{r}}(\tau)=
	q^{2Q(\b r)} \frac{\eta(\tau)^5}{\eta(2\tau)}
	{\rm coeff}_{\left[\zeta_1^{2r_1-r_2},\,\zeta_{2}^{r_1+r_2}\right]} f\left(\b{z};\tau\right)
	\end{align*}
	as claimed in Theorem \ref{TheoremCoeffJacobiExplicit}. In particular, this yields Theorem \ref{TheoremCoeffJacobi}, using Proposition \ref{prop-index} to conclude the transformation properties of $f$.
\end{proof}

\section{Fourier coefficients of additional characters}\label{sec-additionalChars}
In this section we compute the Fourier coefficients of two additional characters from \cite{KW}. Again $p=2$,
Lemma \ref{JacobiExplicit} immediately implies the following special case.
	
	
	

\begin{corollary} \label{cor-00}
We have
	$$\frac{\eta(\tau)^5}{\eta(2 \tau)} {\rm CT}_{[\zeta_1,\,\zeta_2]} f(\b z;\tau)
=\sum_{\substack{ n_1 \geq 0 \\ n_2 \in \mathbb{Z}}} {\rm sgn}^*(n_2)(-1)^{n_1} q^{\frac{n_1(n_1+1)}{2}+n_1n_2+2n_2^2+2n_2+\frac12 },$$
where ${\rm CT}_{[\zeta_1,\,\zeta_2]}:={\rm coeff}_{[\zeta^0_1,\,\zeta_2^0]}$.
\end{corollary}


In addition to characters discussed in (\ref{new}), Kac and Wakimoto also obtained  character formulas for modules $L(-\frac{3}{2} \Lambda_j)$, for $1 \leq j \leq N-1$ \cite[p.130]{KW}.
For $N=3$ we have two additional modules, namely $L(-\frac{3}{2} \Lambda_1)$ and $L(-\frac{3}{2} \Lambda_2)$.  The relevant Fourier coefficients for these characters are
$$
\mathbb{H}_{\b r}(\tau):= -\frac{\eta(\tau)^5}{\eta(2 \tau)} {\rm coeff}_{\left[\zeta_1^{r_1},\,\zeta_2^{r_2}\right]}
 \frac{\vartheta(z_1;2 \tau)\vartheta_{01}(z_2;2 \tau)
\vartheta_{01}(z_1+z_2; 2 \tau)}{\vartheta(z_1; \tau)\vartheta(z_2;
\tau) \vartheta(z_1+z_2; \tau)}, \quad     r_1 \in \frac12 + \Z, \ r_2 \in \Z ,
$$
where the range is $\lvert q\rvert<\lvert \zeta_j\lvert < |1|$, $|q|<\lvert \zeta_1\zeta_2\lvert <1$ and where 
$$
\vartheta_{01}(z; \tau):= \left(q,\zeta q^{\frac12},\zeta^{-1}q^{\frac12};q\right)_\infty.
$$
The next results shows that the Fourier coefficient $\mathbb{H}_{\b r}$ is essentially $\mathbb{G}_{\b \lambda}$ for some $\b \lambda\in\Q^2$.
\begin{proposition} \label{more-char} 
For $p=2$ and for every $r_1 \in \frac12+\mathbb{Z}$, $r_2 \in \mathbb{Z}$, we have 
$$\mathbb{H}_{\b r}(\tau)= q^{-\frac23Q(\b r)}\, \mathbb{G}_{\left(\frac13(r_1+r_2)-\frac12,\frac13(2r_2-r_1)-\frac12\right)}(\tau).$$
\end{proposition}
\begin{proof}
First we conclude from \eqref{T2T} that
\begin{align*}
 \frac{\vartheta_{01}(z;2 \tau)}{\vartheta(z;\tau)}=iq^{-\frac14} \zeta^{\frac12} \left[ \frac{\vartheta (z;2 \tau)}{\vartheta(z;\tau)} \right]_{\zeta \mapsto q^{-1} \zeta}.
\end{align*}
This implies that 
\begin{align*}
&   {\rm coeff}_{\left[\zeta_1^{r_1},\, \zeta_2^{r_2}\right]}
\frac{\vartheta(z_1;2 \tau)\vartheta_{01}(z_2;2 \tau)
\vartheta_{01}(z_1+z_2; 2 \tau)}{\vartheta(z_1; \tau)\vartheta(z_2;
\tau) \vartheta(z_1+z_2; \tau)}\\
& =  -q^{\frac12} {\rm coeff}_{\left[\zeta_1^{r_1},\,\zeta_2^{r_2}\right]}
\left[ \frac{ \zeta_1^{\frac12} \zeta_2 \vartheta(z_1;2 \tau)\vartheta(z_2;2 \tau)
\vartheta (z_1+z_2; 2 \tau)}{\vartheta(z_1; \tau)\vartheta(z_2;
\tau) \vartheta(z_1+z_2; \tau)}\right]_{\zeta_2 \mapsto q^{-1} \zeta_2}  \\
& = -q^{\frac12-r_2} {\rm coeff}_{\left[\zeta_1^{r_1},\,\zeta_2^{r_2}\right]}
\frac{\zeta_2 \zeta_1^{\frac12} \vartheta(z_1;2 \tau)\vartheta(z_2;2 \tau)
\vartheta (z_1+z_2; 2 \tau)}{\vartheta(z_1; \tau)\vartheta(z_2;
\tau) \vartheta(z_1+z_2; \tau)}=
-q^{\frac12-r_2} {\rm coeff}_{\myleft[\zeta_1^{r_1-\frac12},\,\zeta_2^{r_2-1}\myright]}
f(\b z;\tau),
\end{align*} 
where we use \eqref{T2T} to justify that the coefficients do not change under the substitution $\zeta_2 \mapsto q^{-1} \zeta_2$ above. We apply this and Proposition \ref{ConjG} to obtain, as claimed
\begin{align*}
&\mathbb{H}_{\b r}(\tau)= \frac{\eta(\tau)^5}{\eta(2 \tau)} q^{\frac12-r_2} {\rm coeff}_{\myleft[\zeta_1^{r_1-\frac12},\,\zeta_2^{r_2-1}\myright]}
f\left(\b{z};\tau\right) \\
&=q^{\frac12-r_2} q^{-\frac23 Q\left(r_1-\frac12, r_2-1\right)}\, \mathbb{G}_{\frac13 \left(r_1-\frac12+r_2-1,2r_2-2-r_1+\frac12\right)}(\tau) =  q^{-\frac{2}{3} Q(\boldsymbol{r})}\, \mathbb{G}_{\left(\frac13(r_1+r_2)-\frac12,\frac13(2r_2-r_1)-\frac12 \right)}(\tau).
\end{align*}

\vskip-1em
\end{proof}

\section{$q$-hypergeometric formulas and the proof of Theorem \ref{thm:elegant identity}}
In this section,  we study $q$-hypergeometric representations of the series $\mathbb{G}_{\b{r}}$ introduced in Section 3 and in particular prove Theorem \ref{thm:elegant identity}. 
Again we assume that $p=2$. We first require an auxiliary lemma.
\begin{lemma}\label{lem:auxlem} We have for $|q|<|\zeta|^2<1$
	\begin{equation*}		
	\frac{1}{\left(\zeta q^{\frac12},\zeta^{-1} q^{\frac12};q\right)_\infty}=\frac{1}{(q;q)^2_\infty}   \sum_{\substack{n_1\in\mathbb{Z} \\ n_2\geq |n_1|}}  (-1)^{n_1+n_2} q^{\frac{ n_2(n_2+1)}{2}-\frac{n_1^2}{2}}\zeta^{n_1}= \sum_{\substack{n_1 \in \mathbb{Z} \\ n_2\geq 0}} \frac{q^{\frac{|n_1|}2+n_2}}{(q;q)_{n_2} (q;q)_{|n_1|+n_2}} \zeta^{n_1}.
	\end{equation*}
\end{lemma}
\begin{proof}
	The first equality follows directly by setting  $z=\zeta q^{-\frac12}$ in \cite[Lemma 1]{An-1984}.

	For the second equality, we recall \cite[(2.2.5)]{An}, which states that for $|q| < |\zeta|< 1$,
	$$\frac{1}{(\zeta;q)_\infty}=\sum_{n \geq 0} \frac{\zeta^n}{(q;q)_n}.$$
	Using this, we obtain 
	\begin{equation*}\label{identity51}
	\frac{1}{\left(\zeta q^{\frac12},\zeta^{-1} q^{\frac12};q\right)_\infty}=\sum_{\b n\in\N_0^2} \frac{ q^{\frac{n_1+n_2}{2}}}{(q;q)_{n_1} (q;q)_{n_2}}\zeta^{n_1-n_2}.
	\end{equation*}
	The second identity of Lemma \ref{lem:auxlem} then follows by letting $(n_1,n_2)\mapsto(n_1+n_2,n_2)$ if $n_1\geq n_2$ and $(n_1,n_2)\mapsto(n_2,n_2-n_1)$ if $n_1<n_2$.	
\end{proof}

\begin{remark} 
We record another identity similar to the one in Lemma \ref{lem:auxlem}, namely  
\begin{equation}\label{lab2} 
	\frac{1}{\left(\zeta q^{\frac12},\zeta^{-1} q^{\frac12};q\right)_\infty}=\frac{1}{(q;q)_\infty} \sum_{\substack{n_1 \in \mathbb{Z} \\ n_2\geq 0}} \frac{q^{n_2^2+n_2(|n_1|+1)+\frac{|n_1|}2}}{(q;q)_{n_2} (q;q)_{|n_1|+n_2}} \zeta^{n_1}.
	\end{equation}
This equality can be established by proving that the middle expression in Lemma \ref{lem:auxlem} equals the right-hand side of (\ref{lab2})
$$\frac{1}{(q;q)^2_\infty}   \sum_{\substack{\b{n}\in \mathbb{Z}^2 \\ n_2\geq |n_1|}}  (-1)^{n_1+n_2} q^{\frac{ n_2(n_2+1)}{2}-\frac{n_1^2}{2}}\zeta^{n_1}=\frac{1}{(q;q)_\infty} \sum_{\substack{n_1 \in \mathbb{Z} \\ n_2\geq 0}} \frac{q^{n_2^2+n_2(|n_1|+1)+\frac{|n_1|}2}}{(q;q)_{n_2} (q;q)_{|n_1|+n_2}} \zeta^{n_1}.$$
This follows from Ramanujan's identity \cite[p.18, entry 9]{Berndt} after a simple substitution.
For related identities see \cite[Chapter 7]{BM0} and \cite{Wa1}.
\end{remark}



Next we generalize Lemma \ref{lem:auxlem} to ``rank two''. For this, we set
\begin{equation*}
G_{\boldsymbol{r}}(q):=\sum_{\substack{\b n\in \N_0^3 \\ n_4\in\Z}} \frac{q^{n_1+n_2+n_3+ \frac12 \left(|n_4-r_1|+|n_4-r_2|+|n_4|\right)}}{(q;q)_{n_1} (q;q)_{n_1+|n_4-r_1|}(q;q)_{n_2} (q;q)_{n_2+ |n_4-r_2|}  (q;q)_{n_3} (q;q)_{n_3+ |n_4|}}. 
\end{equation*}
Applying Proposition \ref{ConjG}, then Lemma \ref{lem:auxlem}
three times, and 
 \begin{equation}\label{thetafrac}
\frac{\vartheta\left(z;2\tau\right)}{\vartheta\left(z;\tau\right)}
=\frac{q^{\frac1{12}}\eta(2\tau) }{\eta(\tau)\left(\zeta q,\zeta^{-1}q; q^2\right)_\infty},
\end{equation}
immediately implies the following proposition.
\begin{proposition}\label{prop}
	We have, for $p=2$ and $\b r\in\Z^2$,
	\begin{align*}
	&\hspace{-1.5cm} \frac{q^{-\frac14-\frac23Q(\b r)}}{\eta(\tau)^2\eta(2\tau)^2} \mathbb{G}_{\frac13(r_1+r_2,\,2r_2-r_1)}(\tau)  \\
	&={\rm coeff}_{\left[\zeta^{r_1}_1,\,\zeta^{r_2}_2\right]}\frac{1}{\left(\zeta_1 q^{\frac12},\zeta_1^{-1}q^{\frac12},\zeta_2q^{\frac12},\zeta_2^{-1}q^{\frac12},\zeta_1\zeta_2 q^{\frac12},\zeta_1^{-1}\zeta_2^{-1}q^{\frac12};q\right)_\infty}=G_{\boldsymbol{r}}(q).
	\end{align*}
\end{proposition}

Now we are ready to prove the $q$-hypergeometric formula in Theorem \ref{thm:elegant identity}.

\noindent {\bf Proof of  Theorem \ref{thm:elegant identity}:}
Using \eqref{thetafrac} we can rewrite Proposition \ref{prop}  as
\begin{align}\label{eq:Prop52var}
q^{-\frac14} \frac{\eta(\tau)^3}{\eta(2\tau)^3} \operatorname{coeff}_{\left[\zeta_1^{r_1},\,\zeta_2^{r_2}\right]} 	f(\b{z};\tau) =G_{\b{r}}\!\left(q^2\right).
\end{align}
Plugging in the definition of $G_{(0,0)}$, $(q;q)_\infty=q^{-\frac{1}{24}}\eta(\tau)$, and then equation \eqref{eq:Prop52var} gives that the right-hand side of Theorem \ref{thm:elegant identity} equals
\begin{align*}
q^{-\frac12} \frac{\eta(\tau)^5}{\eta(2\tau)} \operatorname{coeff}_{\left[\zeta_1^{0},\,\zeta_2^{0}\right]} 	f(\b{z};\tau).
\end{align*}
The claim now follows by applying Lemma \ref{JacobiExplicit}.
\qed




As a corollary, we find a $q$-hypergeometric type identity for $F(\zeta_1,\zeta_2;q)$. 
\begin{corollary}\label{cor:F} 
We have, for $p=2$,
\begin{align*}
q^{-\frac14 }F\left(\zeta_1,\zeta_2;q\right) 
= \eta(\tau)^2\eta(2\tau)^2
 \sum_{\b r\in\Z^2}  q^{2Q(\b r)}  G_{(2r_1-r_2,\,r_1+r_2)}\left(q^2\right)\zeta_1^{r_1}\zeta_
 2^{r_2}.	
\end{align*}

\end{corollary}

\begin{remark}
As explained in Section 4, Theorem \ref{thm:elegant identity} can be viewed as a rank two analogue of the identity in Lemma \ref{lem:auxlem}.
However, Warnaar  \cite[Section 5]{Wa1} obtained another identity for the false theta with an additional parameter $w$ and no infinite products
\begin{equation*}
\sum_{n \geq 0} \frac{\left(q;q^2\right)_n \left(wq;q^2\right)_n (wq)^n}{(-wq;q)_{2n+1}} =\sum_{n  \geq 0} (-1)^n q^{n^2} w^n.
\end{equation*}
This remarkable identity can be used for explicit computation of radial limits of various partial and false theta functions at root of unity \cite{BM,FOR}.
It would be very interesting to find a similar identity for the rank two false theta functions studied in this paper, presumably with two additional parameters.
\end{remark}



\section{A Jacobi form of index zero and $\lim_{(\zeta_1,\zeta_2) \to (1,1)} F(\zeta_1,\zeta_2;q)$}

In this section we are interested in the limit 
$$F_0(q):=\lim_{(\zeta_1,\zeta_2) \to (1,1)}  F(\zeta_1,\zeta_2;q),$$
which we realize for $p=2$ as a Fourier coefficient of a Jacobi form in two variables.
In \cite[Example 4.3]{BM}, the first and the third author established the following identity for every $p \geq 2$
\[
F_0(q)=\frac{1}{2} \sum_{\b n \in \mathbb{Z}^2} (2n_1-n_2)(2n_2-n_1)(n_1+n_2)q^{pQ\left(\b n-\left(\frac1p,\frac1p\right)\right)}.
\]
Due to the harmonicity of the coefficients in $F_0$, this function is a sum of modular forms of positive integral weight of at most  four \cite[Example 4.3]{BM}. 
As in Section 3, we find the following simplification for $p=2$.
\begin{proposition} For $p=2$,  we have 
$$F_0(q)= \frac14 \sum_{\b n \in \mathbb{Z}^2} (12 n_1 n_2-3n_1^2-3n_2^2-n_1-n_2) q^{ 2 Q\left(\b n-\left(\frac12,\frac12\right)\right) }.$$
\end{proposition}
\begin{proof}
Using the involution $\b n\mapsto -\b n+(1,1)$ (under which $Q\left(\b n-(\frac12,\frac12)\right)$ is invariant), we obtain
\begin{align*}
&2 F_0(q)=  \frac{1}{2} \sum_{\b n \in \mathbb{Z}^2} (2n_1-n_2)(2n_2-n_1)(n_1+n_2)q^{2Q\left(\b n-\left(\frac12,\frac12\right)\right)}  \\
&\hspace{3.8cm} + \frac{1}{2} \sum_{\b n \in \mathbb{Z}^2} (-2n_1+n_2+1)(-2n_2+n_1+1)(-n_1-n_2+2)q^{2Q\left(\b n-\left(\frac12,\frac12\right)\right)}  \\
&= \sum_{\b n\in \mathbb{Z}^2} (-n_1-n_2+1) q^{2Q\left(\b n-\left(\frac12,\frac12\right)\right)}+\sum_{\b n\in \mathbb{Z}^2} \left(6 n_1 n_2-\tfrac32 n_1^2-\tfrac32 n_2^2-\tfrac12 n_1-\tfrac12n_2 \right) 
q^{2Q\left(\b n-\left(\frac12,\frac12\right)\right)}.
\end{align*}
Using again the involution $\b n\mapsto -\b n+(1,1)$ implies that
\begin{equation*}
\sum_{\b n  \in \mathbb{Z}^2}   (n_1+n_2-1) q^{2Q\left(\b n-\left(\frac12,\frac12\right)\right)}=0,
\end{equation*}
yielding the claim.
\end{proof}

To express $F_0$ as the constant term of a multivariable Jacobi form (of index zero and weight 3) we require
\begin{equation*}
J(\b{z};\tau):=\frac{\eta(\tau)^5}{\eta(2 \tau)} \mathcal{T} (\boldsymbol{z};\tau)f(\b z;\tau).
\end{equation*}

Directly from Proposition \ref{prop-index} and the transformation behaviour of $\eta$, we obtain the following statement.
\begin{lemma}
The function $J$ is a Jacobi form of index zero and weight three. To be more precise, we have for $\left(\begin{smallmatrix}
a&b\\ c&d
\end{smallmatrix}\right) \in \Gamma_0(6)$, $\boldsymbol{m} \in 2 \Z^2$, and $\boldsymbol{\ell} \in \Z^2$ 
\begin{align*}
	J\left(\frac{\boldsymbol{z}}{c\tau +d}; \frac{a\tau +b}{c\tau +d}\right)&= \mu \left(\begin{matrix} a & b \\ c & d \end{matrix} \right)  \left(c\tau+d\right)^3 J(\boldsymbol{z};\tau),\\
	J(\boldsymbol{z}+\boldsymbol{m}\tau + \boldsymbol{\ell};\tau)&=J(\boldsymbol{z};\tau),
\end{align*}
where $\mu \left(\begin{smallmatrix} a & b \\ c & d \end{smallmatrix} \right):=\chi \left(\begin{smallmatrix}
	a&b\\c&d
	\end{smallmatrix}\right)^2 \chi \big(\begin{smallmatrix}
		a&2b\\ \frac c2 &d
	\end{smallmatrix} \big)^8$. 
\end{lemma}

The following proposition then indeed realizes $F_0$ as the constant term of $J$.

\begin{proposition} \label{index-zero} For $p=2$, we have 
$$
F_0(q) ={\rm CT}_{\left[\zeta_1,\, \zeta_2\right]} J(\b z;\tau).
$$
\end{proposition}
\begin{proof} By Proposition \ref{ct} and Proposition \ref{ConjG}, we have 
\begin{align*}
	F_0(q)&=\lim_{ (\zeta_1,\,\zeta_2) \to (1,1)}  F(\zeta_1,\zeta_2;q)=\lim_{ (\zeta_1,\,\zeta_2) \to (1,1)}  \sum_{\b n \in \mathbb{Z}^2}  \mathbb{G}_{\b n}(q) \zeta_1^{n_1} \zeta_2^{n_2}= \sum_{\b n \in \mathbb{Z}^2}  \mathbb{G}_{\b n}(q)\\
	&= \frac{\eta(\tau)^5}{\eta(2 \tau)} \sum_{\b n \in \mathbb{Z}^2} q^{2Q(\b n)} {\rm coeff}_{\left[\zeta_1^{n_1+n_2},\,\zeta_2^{2n_1-n_2}\right]}    f(\b z;\tau)\\
	&= \frac{\eta(\tau)^5}{\eta(2 \tau)} {\rm CT}_{\left[\zeta_1,\,\zeta_2\right]} f(\b z;\tau) \sum_{\b n \in \mathbb{Z}^2}   q^{2Q(\b n)} \zeta_1^{-n_1-n_2} \zeta_2^{n_2-2n_1} =  {\rm CT}_{\left[\zeta_1,\,\zeta_2\right]}J({\bf z};\tau).\qedhere
\end{align*}

\end{proof}


\begin{thebibliography}{FOR}
	\bibitem{An-1984} G. Andrews, {\it Hecke modular forms and the Kac--Peterson identities}, Trans. Amer. Math. Soc. {\bf 283} (1984), 451--458.
	
	\bibitem{An} G. Andrews, {\it The theory of partitions}, Cambridge University Press, 1998.
	
	\bibitem{AW} G. Andrews and O. Warnaar, {\em The product of partial theta functions},  Adv. Appl. Math. \textbf{39} (2007), 116--120.
	
	\bibitem{Ad} D. Adamovic, {\em A realization of certain modules for the  $N= 4$ superconformal algebra and the affine Lie algebra  $A_2^{(1)}$}, 
	 Transf. Groups \textbf{21} (2016), 299--327.
	 
	 \bibitem{AM} D. Adamovic and A. Milas, {\em On some vertex algebras related to $V_{-1}(\frak{sl}(n))$ and their characters}, {\tt arXiv:1805.09771}.
	
	 
	\bibitem{Berndt} B. Berndt, {\em Ramanujan's Notebooks}, Part III, Springer.
	
	\bibitem{BCR} K. Bringmann, T. Creutzig, and L. Rolen, {\it Negative index Jacobi forms and quantum modular forms}, Res. Math. Sci. \textbf{1} (2014).
	
	\bibitem{BF} K. Bringmann and A. Folsom, {\it Almost harmonic Maass forms and Kac--Wakimoto characters}, J. Reine Angew. Math. \textbf{694} (2014), 179--202.
	
	\bibitem{BKM} K. Bringmann, J. Kaszian, and A. Milas,  \emph{Higher depth quantum modular forms, multiple Eichler integrals, and $\mathfrak{sl}_3$ false theta functions}, {Res. Math. Sci.}, to appear: arXiv:1704.06891.
	\bibitem{BM0}  K. Bringmann and A. Milas,  {\em W-Algebras, False Theta Functions and Quantum Modular Forms, I}, Int. Math. Res. Not. \textbf{21} (2015), 11351--11387.
	
	\bibitem{BM}  K. Bringmann and A. Milas, {\em W-algebras, higher rank false theta functions, and quantum dimensions}, Selecta Math. {\bf 23} (2017), 1--30.
	
	
	\bibitem{BRZ} K. Bringmann, L. Rolen, and S. Zwegers, {\it On the Fourier coefficients of negative index meromorphic Jacobi forms}, Res. Math. Sci. \textbf{3} (2016).

	 
    \bibitem{C} T. Creutzig, {\em Logarithmic W-algebras and Argyres-Douglas theories at higher rank}, J. High Energy Phys. \textbf{11} (2018):188.
    
	\bibitem{CM} T. Creutzig and A. Milas,  {\em Higher rank partial and false theta functions and representation theory},  Adv. Math. \textbf{314} (2017), 203--227. 
	
	\bibitem{DMZ}  A. Dabholkar, S. Murthy, and D. Zagier, {\em Quantum Black Holes, Wall Crossing, and Mock Modular Forms},  arXiv:1208.4074.
	
	\bibitem{EZ} M. Eichler and D. Zagier, {\em The theory of Jacobi forms}, Birkh\"auser, 1985.
	
	\bibitem{FT} B. Feigin and I. Tipunin,  {\em Logarithmic CFTs connected with simple Lie algebras}, {\tt arXiv:1002.5047}.
	
	\bibitem{FOR} A. Folsom,  K. Ono, and R. Rhoades, {\em Mock theta functions and quantum modular forms}, Forum Math. Pi \textbf{1} (2013).
	
	
	\bibitem{KP} V.  Kac and D.  Peterson, {\em Infinite dimensional Lie algebras, theta functions and modular forms}, Adv. Math. {\bf 53} (1984), 125--264.
	
	\bibitem{KW} V. Kac and M. Wakimoto, {\em A remark on boundary level admissible representations}, C. R. Math. \textbf{355} (2017), 128--132.
	
	 
	\bibitem{Ol} R. Olivetto, {\it Harmonic Maass forms, meromorphic Jacobi forms, and applications to Lie superalgebras}, Ph.D. thesis, University of Cologne, 2016.
  
    \bibitem{Wa1} O. Warnaar, {\em Partial theta functions. I. Beyond the lost notebook},  Proc. London Math. Soc. \textbf{87} (2003), 363--395.
    
    \bibitem{Wa2} O. Warnaar, {\em Partial theta functions}, Ramanujan Encyclopedia, to be published by Springer.
    
    \bibitem{Za} D. Zagier, {\it Periods of modular forms and Jacobi theta functions}, Invent. Math. {\bf 104} (1991), 449--465.
    
    \bibitem{Zw}S. Zwegers, {\em Mock Theta Functions}, Ph.D. thesis, Utrecht University, 2002.
\end{thebibliography}
\end{document}